\documentclass[11pt]{amsart}

\usepackage[margin=0.9in]{geometry}
\usepackage{amsmath,amssymb,amsthm,mathtools}
\usepackage{mathrsfs}
\usepackage{enumitem}
\usepackage{xcolor-material}
\usepackage[unicode, psdextra, colorlinks=true, linkcolor=GoogleBlue, citecolor=GoogleGreen, urlcolor=GoogleBlue, linktocpage]{hyperref}
\usepackage[nameinlink,capitalize]{cleveref}

\newtheorem{theorem}{Theorem}[section]

\newtheorem{proposition}[theorem]{Proposition}
\newtheorem{prop}[theorem]{Proposition}
\newtheorem{corollary}[theorem]{Corollary}
\newtheorem{cor}[theorem]{Corollary}
\newtheorem{lemma}[theorem]{Lemma}

\theoremstyle{definition}

\newtheorem{definition}[theorem]{Definition}
\newtheorem{de}[theorem]{Definition}

\newtheorem{example}[theorem]{Example}

\theoremstyle{remark}
\newtheorem{remark}[theorem]{Remark}
\newtheorem{rmk}[theorem]{Remark}

\crefname{theorem}{Theorem}{Theorems}
\crefname{statement}{Statement}{Statements}
\crefname{proposition}{Proposition}{Propositions}
\crefname{corollary}{Corollary}{Corollaries}
\crefname{lemma}{Lemma}{Lemmas}
\crefname{definition}{Definition}{Definitions}
\crefname{example}{Example}{Examples}
\crefname{remark}{Remark}{Remarks}
\Crefname{statement}{Statement}{Statements}
\Crefname{proposition}{Proposition}{Propositions}
\Crefname{theorem}{Theorem}{Theorems}
\Crefname{corollary}{Corollary}{Corollaries}
\Crefname{lemma}{Lemma}{Lemmas}
\Crefname{definition}{Definition}{Definitions}
\Crefname{example}{Example}{Examples}
\Crefname{remark}{Remark}{Remarks}
\def\PD{\mathrm{PD}}
\def\Z{\mathbb{Z}}
\def\fun{\rightarrow}
\def\iso{\cong}

\def\a{\alpha}
\def\A{\mathbb{A}}
\def\red{\mathrm{red}}
\def\la{\langle}
\def\ra{\rangle}
\def\CC{\C^*}
\def\MZ{Minets--\v{Z}ivanovi\'{c}}
\newcommand{\Hilb}{\operatorname{Hilb}}
\newcommand{\Sym}{\operatorname{Sym}}
\def\C{\mathbb C}
\newcommand{\Q}{\mathbb Q}
\newcommand{\BM}{\operatorname{BM}}
\newcommand{\Gr}{\operatorname{Gr}}
\newcommand{\Span}{\operatorname{Span}}

\title{Perverse filtration on Hilbert schemes via upward flow}

\author{Filip Živanović}
\address{F. T. Živanović,
Simons Center for Geometry and Physics,
Stony Brook, NY 11794-3636, U.S.A.}
\email{fzivanovic@scgp.stonybrook.edu}

\date{}

\begin{document}

\begin{abstract}
We explicitly compute the perverse Leray filtration on the top cohomology of
the Hilbert scheme of points on $\Sigma\times\mathbb{C}$, for any connected smooth
projective curve $\Sigma$. The computation is carried out in the natural
basis given by the $\mathbb{C}^*$-upward-flow cycles. The result is described by a
simple symmetric-function dictionary: upward-flow classes correspond to
products of complete homogeneous symmetric functions, while the
perverse-homogeneous basis corresponds to products of power-sum symmetric
functions. This gives an explicit triangular change-of-basis between the two
bases.
\end{abstract}




\vspace*{-0.03cm}
\maketitle

\setcounter{tocdepth}{1}
\tableofcontents

\section{Introduction}

Perverse Leray filtration, introduced by de Cataldo--Migliorini in \cite{deCataldoMigliorini2005} in the language of perverse sheaves, is one of the basic structures attached to maps of
algebraic varieties.  Given a proper morphism of complex algebraic varieties
\[
        f:X\to Y,
\]
the perverse filtration on \(H^*(X):=H^*(X;\Q)\) measures how the cohomology of
\(X\) is distributed over the base \(Y\).  In the case when $X$ is a moduli
space of Higgs bundles, and $f$ its Hitchin fibration, this filtration plays a particularly important role. Namely, 
\textit{the \(P=W\)
phenomenon}, conjectured in \cite{de2012topology}, and proved in recent years for many general cases \cite{ShenZhang,hausel2022p, maulik2024p},
predicts that the perverse filtration for the Hitchin map is
identified, under nonabelian Hodge theory, with the weight filtration on the
corresponding character variety.

The present paper studies and computes perverse filtration on Hilbert schemes.
Let
\(\Sigma\) be a connected smooth projective curve 
and set
\[
        M_n(\Sigma):=\Hilb^n(\Sigma \times \C).
\]
We consider the map
\[
        h_n:M_n(\Sigma)
        \xrightarrow{\rho}
        \Sym^n(\Sigma\times\C)
        \xrightarrow{\Sym^n(p)}
        \Sym^n(\C),
        \qquad p:\Sigma\times\C\to\C,
\]
where \(\rho\) is the Hilbert--Chow morphism. As shown in \cite{groechenig2014hilbert}, when 
$\Sigma=E$ is an elliptic curve, the space $M_n(E)$ is a particular case of a (parabolic) 
Higgs 
moduli space and $h_n$ is its Hitchin map. 

In the elliptic case, Shen--Zhang proved the \(P=W\) theorem in this
setting and, in particular, computed the ranks of the perverse filtration.  
Denoting by
$\lambda=(\lambda_1,\dots,\lambda_r) \vdash n$ a partition of $n$ and by
$\ell(\lambda)=r$ its length, their result yields:

\begin{prop}[Rank formula, \cite{ShenZhang}]
\label[prop]{prop:intro-sz-rank}
For \(n\le k\le 2n\),
\[
        \dim P_{k}H^{2n}(M_n(\Sigma))
        =
        \#\{\lambda\vdash n\mid \ell(\lambda)\le k-n\}.
\]
\end{prop}

While Shen--Zhang's work does determine the perverse pieces, it does not identify the corresponding subspaces in the natural geometric basis coming from the attracting varieties of the \(\C^*\)-action.  The point of the paper is to supply this missing description, in the {top} cohomological degree.


In order to obtain this explicit description, we use De Cataldo--Migliorini's general recipe 
to recover the perverse filtration 
of a map $f: X\fun Y$ from
restrictions to general affine flags in the base $Y.$ 

\begin{theorem}[dCM flag description, \cite{dCM}]
\label[theorem]{intro thm:dcm-flag}
Consider a proper morphism
\[
        f: X\to Y
\]
with \(X\) smooth and \(Y\) affine. Choosing an embedding
$Y\hookrightarrow \mathbb A^N$
and a sufficiently general affine flag
$\emptyset=\Lambda^{-1}\subset \Lambda^0\subset \Lambda^1\subset\cdots\subset \Lambda^N=\mathbb A^N,$
set
\[
        Y_i:=Y\cap \Lambda^i,
        \qquad
        X_i:=f^{-1}(Y_i).
\]
Then, the perverse filtration is computed via restriction maps:
\[
        P_kH^j(X)
        =
        \ker\left(
        H^j(X)\longrightarrow H^j(X_{j-k-1})
        \right).
\]
\end{theorem}

We make this description completely explicit for the case of the proper morphism
$$h_n: M_n(\Sigma)\fun \Sym^n \C $$
mentioned above, ultimately computing perverse filtration on the top-degree group
       $H^{2n}(M_n(\Sigma)).$
       
The natural basis of this group comes from the \(\C^*\)-action on $M_n(\Sigma)$ induced from the contracting $\C^*$-action on the total space of the trivial line bundle $\Sigma \times \C$
\[
        t\cdot(x,z)=(x,tz)
\]
Following Nakajima \cite{Nak99}, the fixed components of the induced action on \(M_n(\Sigma)\) are
indexed by partitions \(\lambda=(1^{\alpha_1}2^{\alpha_2} \dots n^{\alpha_n})\):
\[
        F_\lambda=\Sym^\lambda \Sigma:=\prod_{i\ge1} \Sym^{\alpha_i} \Sigma.
\]
Given a generic point \(p\in F_\lambda\), its $\C^*$-upward flow
\[
        W_p^+:=\{q\in M_n(\Sigma) \mid \lim_{t\fun 0} t\cdot q = p \}
\]
is a closed subvariety of $M_n(\Sigma)$. This closedness statement generalizes the argument from the recent work of {\MZ} in the elliptic case $\Sigma=E$. 
In that case, these upward-flow subvarieties have a Lagrangian interpretation and were studied as branes of mirror symmetry \cite{MZ}.
 
Denote $W_\lambda^+:=W_p^+$ for a generic point 
$p\in F_\lambda$. It defines Borel--Moore homology classes
\[
        [W_\lambda^+]\in H^{\BM}_{2n}(M_n(\Sigma)
        ),
\]
and hence, by Poincaré duality, cohomology classes
\[
        U_\lambda:=\PD[W_\lambda^+]\in H^{2n}(M_n(\Sigma)).
\]
The classes \(U_\lambda\), for \(\lambda\vdash n\), form a natural geometric
basis of \(H^{2n}(M_n(\Sigma))\).

The geometric mechanism is to turn the abstract restriction maps in \cref{intro thm:dcm-flag} into concrete incidence equations against the collision strata of \(\Sym^n\C\).  Let
\[
        S_\mu\subset \Sym^n\C
\]
be the collision stratum corresponding to a partition \(\mu\vdash n\).
As observed in \cite{MZ}
(for $\Sigma=E$), the restriction  
$$h_n|_{W_\lambda^+}: W_\lambda^+ \fun \Sym^n \C $$
is étale-locally modeled by the symmetrization map
\[
        \prod_i\Sym^{\lambda_i}\C\longrightarrow \Sym^n\C,
        \qquad
        (D_1,\ldots,D_{\ell(\lambda)})\longmapsto D_1+\cdots+D_{\ell(\lambda)},
\]
which is, in particular, a branched cover. 
Its branch count can be written in the symmetric function language. Denoting by $m_\mu$ the 
monomial symmetric function and 
$h_\lambda:=h_{\lambda_1}\cdot \dots\cdot h_{\lambda_{\ell(\lambda)}}$
the product of complete homogeneous symmetric functions, 
we get:

\begin{proposition}
    The number of branches of
\[
       h_n|_{W_\lambda^+}: W_\lambda^+\to \Sym^n\C
\]
over any point of \(S_\mu\) is
precisely $n_\lambda^\mu,$
where $h_\lambda=\sum_{\mu\vdash n} n_\lambda^\mu m_\mu$.
\end{proposition}

Let us explain how this branch count enters the perverse filtration. If
\[
        \Lambda_0\in S_{1^n}\subset \Sym^n\C
\]
is a general point, equivalently, a 0-dimensional affine subspace and
\[
        i_{\Lambda_0}:h_n^{-1}(\Lambda_0)\hookrightarrow M_n(\Sigma)
\]
is the inclusion, then for
\[
        \alpha:=\sum_{\lambda\vdash n}c_\lambda U_\lambda
\]
the restriction \(i_{\Lambda_0}^*\alpha\) is computed by intersecting the
Borel--Moore representatives \(W_\lambda^+\) with the fiber 
\(h_n^{-1}(\Lambda_0)\).
By the branch-count computation above, the intersection consists of \(n_\lambda^{1^n}\) local sheets.  Hence we get 
\[
        \alpha \in P_{2n-1} H^{2n}(M_n(\Sigma))  \overset{\cref{intro thm:dcm-flag}}{\Longleftrightarrow} i_{\Lambda_0}^*\alpha=0
        \quad\Longleftrightarrow\quad
        \sum_{\lambda\vdash n}c_\lambda n_\lambda^{1^n}=0.
\]

This generalizes to the higher-dimensional affine subspaces appearing in the dCM
criterion.  If \(\Lambda_i\subset \Sym^n\C\) is a sufficiently general affine linear
subspace, 
choosing a point in a certain stratum
\[
        y\in \Lambda_i\cap S_\mu,
\]
we pass to the generally singular, yet pure-dimensional fiber
\[
        i_y: h_n^{-1}(y)\hookrightarrow M_n(\Sigma)
\]
and compute its
intersection number with the Borel--Moore representative of $\alpha$:
\[
      \la i_y^*\alpha , [h_n^{-1}(y)_{\red}]\ra =  \sum_{\lambda\vdash n}c_\lambda n_\lambda^\mu.
\]
Here, pure-dimensionality ensures that the Borel--Moore class of the reduced scheme $[h_n^{-1}(y)_{\red}]$ is well-defined. 
Therefore, the kernel of the dCM restriction map is cut out by the vanishing
of precisely these fiber-intersection multiplicities on the collision strata
met by the chosen affine subspace $\Lambda_i$. 
Since a sufficiently general affine subspace of dimension
\(n-r-1\) meets exactly the strata \(S_\mu\) with \(\ell(\mu)>r\), the dCM
flag description gives the following incidence form of the perverse
filtration.

\begin{prop}
\label[proposition]{prop:intro-incidence}
Let
\[
        \alpha=\sum_{\lambda\vdash n}c_\lambda U_\lambda
        \in H^{2n}(M_n(\Sigma)).
\]
Then
\[
        \alpha\in P_kH^{2n}(M_n(\Sigma)) \iff
        \sum_{\lambda\vdash n}c_\lambda n_\lambda^\mu=0,\  \forall \mu\vdash n \text{ satisfying }\ell(\mu)>k-n.
\]
\end{prop}
In order to find a nice perverse-homogeneous basis, the problem  
becomes diagonalizing the collision-incidence matrix
\(N=(n_\lambda^\mu)\).
The appearance of products of Newton power sums 
$$p_\lambda:=p_{\lambda_1} \cdots  p_{\lambda_{\ell(\lambda)}}$$
is not an arbitrary change of basis: the
power-sum basis is precisely the symmetric-function basis adapted to these
incidence conditions, since the monomial expansion of \(p_\lambda\) has no
terms supported on more than \(\ell(\lambda)\) distinct variables.  Thus the
length filtration on partitions is diagonalized by the power sums.  This is
classical symmetric-function theory.  Let
\(\Lambda^n\) denote the degree-\(n\) part of the ring of symmetric
functions.  Since the \(h_\mu\)'s form a basis of \(\Lambda^n\),
there is a unique linear isomorphism
\[
        \psi:\Lambda^n\xrightarrow{\;\iso\;}H^{2n}(M_n(\Sigma)),
        \qquad h_\mu\longmapsto U_\mu.
\]
We define
\[
        \eta_\lambda:=\psi(p_\lambda)
        =\sum_{\mu\vdash n}a_{\lambda\mu}U_\mu,
\]
where 
$$p_\lambda=\sum_\mu a_{\lambda\mu}h_\mu$$ is the expansion of Newton power sums in the complete-homogeneous basis.  The coefficients
\(a_{\lambda\mu}\in \Z\) are \textit{upper-triangular} with respect to length:
\begin{equation}\label{intro: eq uppertriangularity}
        a_{\lambda\mu}=0
        \qquad\text{unless}\qquad
        \mu=\lambda
        \text{ or }
        \ell(\mu)>\ell(\lambda),
\end{equation}
Moreover,
\[
        a_{\lambda\lambda}=\prod_i\lambda_i.
\]

Finally, we get to the main theorem of this paper.  

\begin{theorem}
\label[theorem]{thm:intro-main}
For every \(k\),
\[
        P_kH^{2n}(M_n(\Sigma))
        =
        \Span\{\eta_\lambda\mid \ell(\lambda)\le k-n\}.
\]
\end{theorem}

Our contribution is twofold.  First, a \textit{method}: we convert the de
Cataldo--Migliorini flag description into an explicit incidence calculation
against the collision strata of \(\Sym^n\C\) in the upward-flow basis, and we
solve the resulting equations by a single, classical change of basis between
the complete homogeneous symmetric functions and the Newton power sums.  This
produces the dictionary
\[
        \psi:\ h_\lambda\longmapsto U_\lambda,
        \qquad p_\lambda\longmapsto \eta_\lambda,
\]
under which the perverse filtration corresponds to the filtration of
\(\Lambda^n\) by partition length.  

Second, an \textit{upward-flow presentation}: while Shen--Zhang describe the perverse filtration in terms of tautological classes, we
identify its steps as concrete spans of upward-flow cycles, 
for every step of the filtration and for an arbitrary connected smooth projective curve $\Sigma$.

An interesting consequence of \cref{thm:intro-main} and \eqref{intro: eq uppertriangularity} is a clean triangularity between the geometric and perverse-homogeneous bases:

\begin{corollary}
The perverse-homogeneous basis \(\{\eta_\lambda\}_{\lambda\vdash n}\) of
\(H^{2n}(M_n(\Sigma))\) has an explicit upper-triangular transition matrix
from the upward-flow basis \(\{U_\lambda\}_{\lambda\vdash n}\), with respect to
any total order on partitions refining the length order.
\end{corollary}

The triangular shape of this matrix is consistent with the \textit{generic} linear-algebra picture: a generic ordered basis determines a complete flag in opposite position to a fixed filtration, so an adapted basis for the filtration has an anti-triangular echelon form, up to ordering conventions. Here, the filtration and the ordered basis are both \textit{fixed} by the geometry, so this generic relative-position pattern is not automatic.

Whether explicit triangular perverse-to-upward-flow transition matrices can be found for other Higgs moduli spaces seems a worthwhile avenue for further research.

\medskip
\noindent\textbf{Acknowledgments.}
I thank Mark de Cataldo for helpful conversations.
I am especially grateful to Alexandre Minets for explaining the details of Shen--Zhang's work and for our ongoing collaboration, which gave rise to the questions studied in this paper.
I acknowledge the hospitality and support of the Simons Center for Geometry and Physics.

\section{Algebra of symmetric functions}

Let
\[
\Lambda:=\varprojlim_{n\fun \infty }\Z[x_1,x_2,x_3,\ldots,x_n]^{S_n}
\]
be the ring of symmetric functions over $\Z$, given as the inverse limit of the rings of symmetric polynomials in
finitely many variables. 
We write \(\Lambda^n\) for its degree \(n\)
homogeneous part. We denote the same ring over $\Q$ by
$$\Lambda_\Q:=\Lambda\otimes \Q$$
and by $\Lambda_\Q^n$ its degree \(n\)
homogeneous part.
For background on these rings and the standard statements about them used here, we
refer to \cite[Ch.~I]{Macdonald}.

We recall some notation for partitions.  A \textit{partition} of \(n\) will be written
\[
        \lambda=(\lambda_1,\ldots,\lambda_r)\vdash n,
        \qquad
        \lambda_1\ge \cdots\ge \lambda_r>0,
        \qquad
        \sum_i\lambda_i=n.
\]
Its \textit{length} is
\[
        \ell(\lambda)=r.
\]
We also use the notation
\[
        \lambda=(1^{\alpha_1}2^{\alpha_2}\cdots),
\]
where \(\a_j=\a_j(\lambda)\) is the multiplicity of the part \(j\) in
\(\lambda\).

Given a
partition \(\mu=(\mu_1,\ldots,\mu_s) \vdash n\), recall the \textit{monomial symmetric function} \(m_\mu\)
\[
        m_\mu
        =
        \sum_{\substack{1\le i_1<\cdots<i_s\\
        \sigma\in S_s/\operatorname{Stab}(\mu)}}
        x_{i_1}^{\mu_{\sigma(1)}}\cdots x_{i_s}^{\mu_{\sigma(s)}} \in \Lambda^n,
\]
where \(\operatorname{Stab}(\mu)\subset S_s\) is the subgroup preserving
\((\mu_1,\ldots,\mu_s)\). Equivalently, \(m_\mu\) is the sum of all
distinct monomials obtained by assigning the parts of \(\mu\) to distinct
variables.  For example,
\[
        m_{21}=\sum_{i\ne j}x_i^2x_j,
        \quad
        m_{111}=\sum_{i<j<k}x_ix_jx_k.
\]
It is a standard statement that 
\begin{prop}
    The set $\{ m_\mu \mid \mu \vdash n\}$ forms a $\Z$-module basis of $\Lambda^n$.
\end{prop}
Thus, every homogeneous symmetric function 
\(f\) of degree \(n\) can be written
uniquely as
\[
        f=\sum_{\mu\vdash n} b_\mu m_\mu.
\]
The coefficient of \(m_\mu\) in this expansion we denote by 
\[
        [m_\mu]f:=b_\mu \in \Z
\]

The \textit{Newton power sums} are
\[
        p_k:=\sum_i x_i^k \in \Lambda^k.
\]
For a partition \(\lambda=(\lambda_1,\ldots,\lambda_r)\), set
\[
        p_\lambda:=p_{\lambda_1}\cdots p_{\lambda_r}.
\]

\begin{lemma}[Support length of power sums]\label[lemma]{lem:p-support-length}
Let \(\lambda,\mu\vdash n\). Then
\[
        [m_\mu]p_\lambda=0
        \qquad\text{whenever}\qquad
        \ell(\mu)>\ell(\lambda).
\]
\end{lemma}

\begin{proof}
Write \(\lambda=(\lambda_1,\ldots,\lambda_r)\).  A monomial appearing in
\[
        p_\lambda=p_{\lambda_1}\cdots p_{\lambda_r}
\]
has the form
\[
        x_{i_1}^{\lambda_1}\cdots x_{i_r}^{\lambda_r}.
\]
After collecting equal indices, such a monomial involves at most \(r\)
distinct variables.  Hence no monomial of type \(\mu\) can occur if
\(\ell(\mu)>r=\ell(\lambda)\).
\end{proof}

The \textit{complete homogeneous symmetric function} \(h_k\) is
\[
        h_k:=\sum_{i_1\le\cdots\le i_k}x_{i_1}\cdots x_{i_k} \in \Lambda^k
\]
For example
\[
        h_1=\sum_i x_i, \quad 
        h_2=\sum_i x_i^2+\sum_{i<j}x_ix_j, \quad
        h_3
        =
        \sum_i x_i^3
        +
        \sum_{i\ne j}x_i^2x_j
        +
        \sum_{i<j<k}x_ix_jx_k.
\]
For a partition \(\lambda=(\lambda_1,\ldots,\lambda_r)\), set
\[
        h_\lambda:=h_{\lambda_1}\cdots h_{\lambda_r}.
\]

\begin{definition}\label[definition]{def:n-lambda-mu}
For partitions \(\lambda,\mu\vdash n\), define the integers
\(n_\lambda^\mu\) by the expansion
\[
        h_\lambda=\sum_{\mu\vdash n} n_\lambda^\mu m_\mu .
\]
\end{definition}

\begin{lemma}\label[lemma]{lem:matrix-count}
The coefficient \(n_\lambda^\mu\) is the number of nonnegative integer
matrices \(A=(a_{ij})\) with row sums \(\lambda_i\) and column sums \(\mu_j\):
\[
        \sum_j a_{ij}=\lambda_i,
        \qquad
        \sum_i a_{ij}=\mu_j.
\]
\end{lemma}

\begin{proof}
Write
\[
h_\lambda=h_{\lambda_1}\cdots h_{\lambda_r}.
\]
A monomial appearing in \(h_{\lambda_i}\) has total degree \(\lambda_i\).
When multiplying the factors \(h_{\lambda_i}\), the exponent of a variable
records how much of the total degree is assigned to that variable by each
factor.

Equivalently, choosing a monomial contribution to \(h_\lambda\) amounts to
choosing nonnegative integers \(a_{ij}\), where \(a_{ij}\) is the contribution
of the factor \(h_{\lambda_i}\) to the exponent of the \(j\)-th variable.
Since the \(i\)-th factor has total degree \(\lambda_i\), we have
\[
        \sum_j a_{ij}=\lambda_i.
\]
If the resulting monomial has type \(\mu\), then, after grouping variables by
their final exponents, the column sums are
\[
        \sum_i a_{ij}=\mu_j.
\]
Thus the coefficient of \(m_\mu\) in \(h_\lambda\) is exactly the number of
such matrices.
\end{proof}

In particular when $\mu=1^n$, one can explicitly compute:
\begin{cor}\label[cor]{cor: fla za nlambda1n}
   \[
        n_\lambda^{1^n}
        =
        \frac{n!}{\prod_i\lambda_i!}.
\]
\end{cor}

\begin{remark}\label[remark]{rem:labelled-columns}
In \cref{lem:matrix-count}, the columns are labeled by the chosen support
points of a monomial of type \(\mu\).  This matters when \(\mu\) has repeated
parts.  For example, for \(\mu=(1,1,1)\), the three columns correspond to
three distinct variables \(x_i,x_j,x_k\), even though the column sums are
equal.  Thus one obtains
        $n_{111}^{111}=6,$
not \(1\).
\end{remark}

\begin{example}[The coefficients \(n_\lambda^\mu\) for \(n=3\)]
For \(n=3\), order the partitions by
\[
        3,\quad 21,\quad 111.
\]
The coefficients \(n_\lambda^\mu\), with rows indexed by
\(\mu\) and columns indexed by \(\lambda\), are
\[
\begin{array}{c|ccc}
\mu\backslash\lambda
&3&21&111\\
\hline
3   &1&1&1\\
21  &1&2&3\\
111 &1&3&6
\end{array}
\]
Equivalently,
\[
        h_3=m_3+m_{21}+m_{111},
\]
\[
        h_{21}=h_2h_1=m_3+2m_{21}+3m_{111},
\]
\[
        h_{111}=h_1^3=m_3+3m_{21}+6m_{111}.
\]
Thus, for example,
\[
        n_{21}^{111}=3,
\]
because there are three ways to distribute three distinct simple points into
packets of sizes \((2,1)\).
\end{example}

Another standard statement is that complete symmetric functions make a basis:

\begin{lemma}\label[lemma]{lem:h-basis}
The set $\{ h_\mu \mid \mu \vdash n\}$ forms a $\Z$-module basis of $\Lambda^n$.
\end{lemma}

Thus, this makes the following definition valid: 
\begin{definition}\label[definition]{def:a-lambda-mu}
For partitions \(\lambda,\mu\vdash n\), define \(a_{\lambda\mu}\in\Z\) by
the unique expansion
\[
        p_\lambda=\sum_{\mu\vdash n}a_{\lambda\mu}h_\mu.
\]
\end{definition}

\begin{proposition}\label[proposition]{prop:a-triangular}
The coefficients \(a_{\lambda\mu}\) satisfy
\[
        a_{\lambda\mu}=0
        \qquad\text{unless}\qquad
        \mu=\lambda
        \text{ or }
        \ell(\mu)>\ell(\lambda).
\]
Moreover,
\[
        a_{\lambda\lambda}=\prod_i\lambda_i,
        \qquad
        a_{\lambda,1^n}=(-1)^{n-\ell(\lambda)}.
\]
\end{proposition}

\begin{proof}
By expanding the logarithmic form of Newton's identities
\[
        \sum_{m\ge1}\frac{p_m}{m}t^m
        =
        \log\left(1+\sum_{r\ge1}h_rt^r\right),
\]
which follows from \cite[Chapter~I, \S2, (2.10)]{Macdonald}, one obtains
\[
        p_m
        =
        \sum_{\rho\vdash m}
        (-1)^{\ell(\rho)-1}
        \frac{m(\ell(\rho)-1)!}{\prod_j m_j(\rho)!}
        h_\rho .
\]

In particular, the length-one term is
\[
        m h_m.
\]
All other terms have length strictly larger than \(1\).  Thus
\[
        p_m
        =
        m h_m
        +
        \sum_{\substack{\rho\vdash m\\ \ell(\rho)>1}}
        b_{m\rho}h_\rho
\]
for some coefficients \(b_{m\rho}\).

Now let
       $ \lambda=(\lambda_1,\ldots,\lambda_r), r=\ell(\lambda).$
Then
\[
        p_\lambda
        =
        p_{\lambda_1}\cdots p_{\lambda_r}.
\]
In this product, the only way to obtain a term of total length exactly \(r\)
is to take the length-one term
$\lambda_i h_{\lambda_i}$
from every factor \(p_{\lambda_i}\).  This gives
\[
        \left(\prod_i\lambda_i\right)
        h_{\lambda_1}\cdots h_{\lambda_r}
        =
        \left(\prod_i\lambda_i\right)h_\lambda.
\]
If, from at least one factor, we choose a term \(h_\rho\) with
        $\ell(\rho)>1,$
then the resulting product has length strictly larger than \(r\).  Therefore
there is no equal-length mixing: if
\[
        \ell(\mu)=\ell(\lambda),
\]
then
\[
        a_{\lambda\mu}=0
        \qquad\text{unless}\qquad
        \mu=\lambda.
\]
This also proves
\[
        a_{\lambda\lambda}=\prod_i\lambda_i.
\]

Finally, the coefficient of \(h_{1^n}\) in \(p_\lambda\) is obtained by
taking the coefficient of \(h_{1^{\lambda_i}}\) in each factor
\(p_{\lambda_i}\).  From the one-part formula,
       $ [h_{1^m}]p_m=(-1)^{m-1}.$
Hence
\[
        a_{\lambda,1^n}
        =
        \prod_i (-1)^{\lambda_i-1}
        =
        (-1)^{\sum_i\lambda_i-\ell(\lambda)}
        =
        (-1)^{n-\ell(\lambda)}. \qedhere
\]
\end{proof}

This immediately yields
\begin{cor}[Upper-triangularity]
For any $n$, the matrix $$A_n:=(a_{\lambda,\mu})_{\lambda,\mu\vdash n}$$
is upper-triangular with respect to any total order that refines the partial order given by partition length.
Moreover, as its diagonal entries are non-zero, it is nondegenerate, $\det A_n\neq 0.$
\end{cor}

\begin{example}\label[example]{ex:n-equals-3}
For \(n=3\), we have
\[
        p_3=3h_3-3h_{21}+h_{111},
\]
\[
        p_{21}=p_2p_1=(2h_2-h_{11})h_1=2h_{21}-h_{111},
\]
\[
        p_{111}=h_{111}.
\]
Thus, 
the change-of-basis matrix is
\[
\begin{pmatrix}
p_3\\
p_{21}\\
p_{111}
\end{pmatrix}
=
\begin{pmatrix}
3&-3&1\\
0&2&-1\\
0&0&1
\end{pmatrix}
\begin{pmatrix}
h_3\\
h_{21}\\
h_{111}
\end{pmatrix}.
\]
\end{example}

    Unlike $\{m_\mu\}_{\mu\vdash n}$ and $\{h_\mu\}_{\mu\vdash n}$, the set 
    $\{p_\mu\}_{\mu\vdash n}$
    is \textit{not} a $\Z$-module basis of $\Lambda^n$, since the matrix $A_n$ has $\det\neq \pm 1$. Nevertheless, 
    as the same matrix is nondegenerate, we have

\begin{lemma}\label[lemma]{lem:p-basis}
The set $\{ p_\mu \mid \mu \vdash n\}$ forms a $\Q$-vector space  basis of $\Lambda_\Q^n$.
\end{lemma}
\section{Hilbert schemes and upward flow}

From now on fix a connected smooth projective curve \(\Sigma\), and write
\[
        M_n=M_n(\Sigma):=\Hilb^n(\Sigma\times\C).
\]
Let
\[
        p:\Sigma\times\C\to\C
\]
be the projection, and let
\[
        h_n:M_n\fun \Sym^n(\Sigma \times \C)\fun \Sym^n\C
\]
be the composition of the Hilbert--Chow morphism with \(\Sym^n(p)\).  The
\(\C^*\)-actions
\[
        t\cdot(x,z)=(x,tz),\  t\cdot z = tz
\]
on \(\Sigma\times\C\) and $\C$ induce actions on \(M_n\), and $\Sym^n \C$, making the map $h_n$ equivariant.

The $\CC$-fixed locus on $M_n$ is
indexed by partitions \(\lambda\vdash n\).  If
\[
        \lambda=(1^{\alpha_1}2^{\alpha_2}\cdots),
\]
then the corresponding fixed component is
\[
        F_\lambda=\Sym^\lambda \Sigma
        :=
        \prod_{r\ge1}\Sym^{\a_r}(\Sigma).
\]
This is explained in \cite[Ch.7]{Nak99} for $\Hilb^n(T^*\Sigma)$, but the same proof works for any line bundle on a curve, hence in particular, for the trivial bundle $\Sigma \times \C$ that we are considering here.

The $\C^*$-action on $M_n$ is \textit{contracting}, meaning that every point has the limit
$$\forall p\in M_n \ \ \ \exists \lim_{t\to0}t\cdot p$$

The subvariety of points that also have the other limit is called 
\textit{the core }
$$C_n:=\{q\in M_n \mid \exists \lim_{t\to \infty} t\cdot q\}$$
In our setup, since $h_n$ is equivariant and the $\C^*$-action contracts $\Sym^n \C$ to a point, we have equality of sets:
$$C_n=h_n^{-1}(0)$$

For a fixed
component $F_\lambda$, denote the downward flow and its closure in the core:
\[
        W^-_{F_\lambda}
        :=
        \{q\in M_n\mid \lim_{t\to\infty}t\cdot q\in F_\lambda\}, \ 
        C_\lambda:=\overline{W^-_{F_\lambda}}\subset C_n.
\]

\begin{lemma}[Core components]\label[lemma]{lem:core-components}
The subvarieties $C_\lambda$, for $\lambda\vdash n$, are precisely the
irreducible components of the core
        $C_n$. Thus,     
their fundamental classes $[C_\lambda]$ form a basis of
        $H_{2n}(C_n)\iso H_{2n}(M_n)$.
\end{lemma}

\begin{proof}
The downward Białynicki--Birula pieces are contained in the attracting core of
the semiprojective action.  Their closures are irreducible, since each
$W^-_{F_\lambda}$ is an affine bundle over the irreducible fixed component
$F_\lambda$.  These closures cover $C_n$.

Let
\[
        \lambda=1^{\alpha_1}2^{\alpha_2}\cdots .
\]
Then
\[
        F_\lambda\simeq\prod_{r\ge1}\Sym^{\alpha_r}\Sigma,
        \qquad
        \dim F_\lambda=\sum_{r\ge1}\alpha_r=\ell(\lambda).
\]
We claim that the fiber of
\[
        W^-_{F_\lambda}\longrightarrow F_\lambda
\]
has dimension
\[
        n-\ell(\lambda)=\sum_{r\ge1}\alpha_r(r-1).
\]
Indeed, over a point of the open locus where the support points are pairwise
distinct, the Hilbert scheme is locally the product of the corresponding
punctual Hilbert schemes.  A point of $F_\lambda$ consists, for each
part-size $r$, of $\alpha_r$ punctual fixed subschemes of length $r$, each
supported at a distinct point $(x,0)$ on the corresponding vertical curve
$\{x\}\times\C$.  The downward
piece in each punctual length-$r$ factor is the punctual Hilbert scheme
$\Hilb^r_0(\A^2)$, whose dimension is $r-1$ by Briançon's theorem.  Hence the
downward fiber dimension is the sum of the punctual dimensions,
\[
        \sum_{r\ge1}\alpha_r(r-1)=n-\ell(\lambda).
\]
Therefore
\[
        \dim W^-_{F_\lambda}
        =\dim F_\lambda+n-\ell(\lambda)
        =n.
\]
Thus all $C_\lambda=\overline{W^-_{F_\lambda}}$ have dimension $n$.  Since
they cover the
core 
$C_n=h_n^{-1}(0)$, 
they are
precisely its irreducible components.

Since $C_n$ is projective, its ordinary and Borel--Moore homology agree in
top degree.  The top homology of a pure-dimensional complex variety is freely
generated by the fundamental classes of its irreducible components, giving the
basis $\{[C_\lambda]\}$.  Finally, in much higher generality, the core is a
deformation retract of the total space; see \cite[Prop.~3.14]{RZ1}. 
Thus $H_{2n}(C_n)\cong H_{2n}(M_n)$.
\end{proof}

For a point \(p\in F_\lambda\), let \(W_p^+\) denote the
upward $\C^*$-flow subvariety through \(p\):
\[
        W_p^+
        =
        \{q\in M_n\mid \lim_{t\to0}t\cdot q=p\}.
\]

\begin{de}[Generic locus in \(F_\lambda\)]
\label[definition]{prop:generic-very-stable}

Let
       $ F_\lambda^\circ\subset F_\lambda$
be the open locus consisting of tuples
\[
        (D_i)_{i\ge1}\in \prod_{i\ge1}\Sym^{\alpha_i}\Sigma
\]
such that each \(D_i\) is reduced and the supports of the divisors \(D_i\)
are pairwise disjoint. 

\end{de}

The following result identifies this locus with the \textit{very-stable locus} of $F_\lambda$, i.e. the points $p$ such that $W_p^+$ is closed. 
The implication "$\Rightarrow$" generalizes the argument of 
\cite[Prop.~4.13]{MZ}; the converse is a local shear argument.

\begin{prop}[Very-stable locus in \(F_\lambda\)]\label[prop]{prop:very-stable-locus}
Let \(p\in F_\lambda\).  Then
\[
        p\in F_\lambda^\circ
        \ \Longleftrightarrow
        \ W_p^+\cap C_n=\{p\} \ \Longleftrightarrow \ W_p^+ \text{ is closed}.
\]
\end{prop}

\begin{proof}
Assume first that \(p\in F_\lambda^\circ\).  Since the support points of
\(p\) are pairwise distinct, the Hilbert scheme is, étale locally near
\(p\), the product of the corresponding punctual Hilbert schemes.  Under
this product description, the upward flow \(W_p^+\) is the product of the
upward flows of the vertical punctual ideals.  The proof is therefore
identical to the proof for \(\Hilb^n(E\times\C)\) given in
\cite[Prop.~4.13]{MZ}: the single-support case is treated there by the
vertical embedding, and the general distinct-support case follows from the
étale local product decomposition.  The same argument applies here because it
only uses the local product structure of the Hilbert scheme near distinct
support points and the vertical curve
\(\{x\}\times\C\subset\Sigma\times\C\).  Hence
\[
        W_p^+\cap C_n=\{p\}.
\]
By the very-stability criterion \cite[Prop.~2.20]{MZ}, this is equivalent to
closedness of \(W_p^+\).

Conversely, suppose \(p\notin F_\lambda^\circ\).  Then at least two support
points in the product presentation
\[
        F_\lambda\simeq \prod_{i\ge1}\Sym^{\alpha_i}\Sigma
\]
collide.  Choose local coordinates \(u\) on \(\Sigma\) and \(z\) on \(\C\)
near the corresponding point of \(\Sigma\times\C\), so that the
\(\C^*\)-action is
\[
        t\cdot(u,z)=(u,tz).
\]
Let
\[
        A_p=\C\{u,z\}/I_p
\]
be the local algebra of \(p\) at this support point.  Since \(p\) lies on
the boundary of \(F_\lambda^\circ\), the horizontal coordinate \(u\) is
nonzero in \(A_p\).  Indeed, if the class of \(u\) vanished in every local
factor, then the local pieces would be purely vertical and supported at
pairwise distinct points, which is precisely the locus \(F_\lambda^\circ\).

For \(a\ne0\), apply the local shear
\[
        z\longmapsto z+a u
\]
to obtain a nearby Hilbert-scheme point \(q_a\).  Since both \(u\) and \(z\)
are nilpotent in the local algebra of \(p\), the vertical coordinate remains
nilpotent in the local algebra of \(q_a\).  Hence \(q_a\in C_n=h_n^{-1}(0)\).
Moreover \(q_a\ne p\), because \(u\ne0\) in \(A_p\).

Under the \(\C^*\)-action, the local equation \(z+a u\) becomes
\(z+t a u\).  Therefore
\[
        \lim_{t\to0} t\cdot q_a=p.
\]
Thus \(q_a\in W_p^+\cap C_n\) and \(q_a\ne p\).  Hence
\[
        W_p^+\cap C_n\ne\{p\},
\]
so \(p\) is not very stable.  This proves the converse.
\end{proof}

Thus given $p\in F_\lambda^\circ$, $W_p^+$ is a closed submanifold of $M_n$, hence its Borel--Moore homology class 
$$[W_p^+]\in H^{\BM}_{2n}(M_n)$$ is well-defined. 

\begin{remark}[The elliptic case]\label[remark]{rem:elliptic-symplectic}
If \(\Sigma=E\) is an elliptic curve, then
\[
        E\times\C\simeq T^*E
\]
is holomorphic-symplectic with its standard cotangent symplectic structure.  
Hence
\[
        M_n(E)=\Hilb^n(T^*E)
\]
is holomorphic symplectic too, and the very stable 
upward-flow subvarieties \(W_p^+\)
are holomorphic Lagrangians. These were considered in \cite{MZ} as the 
branes on the one side of mirror symmetry for Hilbert schemes of symplectic surfaces, out of which $T^*E$ is the particular case.

\end{remark}

\begin{lemma}
For any choices of $p\in F_\lambda^\circ$,
the classes $\{[W_p^+]\}_{\lambda\vdash n}$ are dual to the 
$\{[C_\mu]\}_{\mu\vdash n}$
i.e. 
$$[W_\lambda^+]\cdot [C_\mu]=\delta_{\lambda\mu}.$$
\end{lemma}

\begin{proof}

By
\cref{prop:very-stable-locus},
\[
        W_p^+\cap C_n=\{p\}.
\]
Since $p\in F_\lambda\subset C_\lambda$, and 
$p$ is a very stable point of
$F_\lambda$, it does not lie on any other component $C_\mu$.  
Indeed, otherwise it would be a limit 
$\lim_{t\fun 0} t\cdot q =p$ of $q\in W_{F_\mu}^-$, i.e. 
$q\in W_p^+$ which contradicts the equivalent 
 very stability condition $W_p^+ \cap C_n=\{p\}$.

At the point $p$, the upward and downward Białynicki--Birula directions are
opposite weight spaces for the $\C^*$-action, hence meet transversely.  With
complex orientations, the local intersection multiplicity is $1$.  Therefore
\[
        [W_\lambda^+]\cdot [C_\mu]=\delta_{\lambda\mu}. \qedhere
\]
\end{proof}

As by \cref{lem:core-components} the $[C_\mu]$ form the basis of $H_{2n}(M_n)$, the previous lemma makes
the following
well-defined: 

\begin{de}
For arbitrary point $p\in F_\lambda^{\circ}$, denote 
    $$[W_\lambda^+]:=[W_p^+]\in H^{\BM}_{2n}(M_n)$$
Since \(M_n\) is
smooth, Poincaré duality identifies
       $ H^{\BM}_{2n}(M_n)\overset{\operatorname{PD}}{\cong} H^{2n}(M_n),$
and we denote
\[
        U_\lambda:=\PD[W_\lambda^+]\in H^{2n}(M_n)
\]
for the corresponding cohomology class.
\end{de}

\begin{cor}\label[corollary]{lem:upward-flow-basis}
     $\{U_\lambda \mid \lambda\vdash n\}$ form a basis of $H^{2n}(M_n).$
\end{cor}
\section{Perverse filtrations and the flag description}

Let
\[
        f:X\to Y
\]
be a proper morphism of complex algebraic varieties, with \(X\) smooth of
 dimension \(d_X\) and \(Y\) quasi-projective.  We use the shifted convention
for the perverse Leray filtration, normalized by the defect of semismallness
\[
        r(f):=\dim(X\times_Y X)-\dim X .
\]
Thus the filtration is defined using the shifted complex
\(Rf_*\Q_X[d_X]\):
\[
        P_kH^j(X)
        :=
        \operatorname{Im}\left(
        \mathbb H^{j-d_X}
        \left(Y,{}^p\tau_{\le k-r(f)}Rf_*\Q_X[d_X]\right)
        \longrightarrow
        \mathbb H^{j-d_X}
        \left(Y,Rf_*\Q_X[d_X]\right)=H^j(X)
        \right).
\]
With this normalization, the
perverse filtration is concentrated in the interval \([0,2r(f)]\).  If \(f\) is
a fibration with equidimensional fibers and \(d_Y=\dim Y\), then
\[
        r(f)=d_X-d_Y,
\]
and the same definition can be written in the simpler form
\[
        P_kH^j(X)
        =
        \operatorname{Im}\left(
        \mathbb H^j
        \left(Y,{}^p\tau_{\le k+d_Y}Rf_*\Q_X\right)
        \longrightarrow
        H^j(X)
        \right).
\]

We will use the following theorem of de Cataldo--Migliorini (dCM)

\begin{theorem}[dCM flag description {\cite{dCM}}]
\label[theorem]{thm:dcm-flag}
Let
\[
        f:X\to Y
\]
be a proper morphism of complex algebraic varieties, with \(X\) smooth of
dimension \(d_X\) and \(Y\) affine.  Choose an embedding
\[
        Y\hookrightarrow \mathbb A^N
\]
and a sufficiently general affine flag
\[
        \emptyset =\Lambda^{-1}\subset \Lambda^0\subset \Lambda^1\subset\cdots\subset \Lambda^N=\mathbb A^N,
        \qquad \dim \Lambda^i=i.
\]
Set
\[
        Y_i=Y\cap \Lambda^i,
        \qquad
        X_i=f^{-1}(Y_i).
\]
Then, the perverse filtration is computed via restriction maps:
\[
        P_kH^j(X)
        =
        \ker\left(
        H^j(X)\longrightarrow H^j(X_{j-k-1})
        \right).
\]
\end{theorem}

\begin{definition}[Collision stratification of \(\Sym^n\C\)]
\label[definition]{def:collision-stratification}
For a partition \(\mu=(\mu_1,\ldots,\mu_s)\vdash n\), let
\[
        S_\mu\subset \Sym^n\C
\]
be the locus of effective divisors on \(\C\) with multiplicity pattern
\(\mu\):
\[
        S_\mu
        =
        \left\{
        \mu_1x_1+\cdots+\mu_sx_s
        \;\middle|\;
        x_i\in\C,\ x_i\neq x_j\text{ for }i\neq j
        \right\}.
\]
The order of the parts of \(\mu\) is irrelevant.  These locally closed smooth
strata give the \textit{collision stratification}
\[
        \Sym^n\C=\bigsqcup_{\mu\vdash n}S_\mu .
\]
Moreover
\[
        \dim S_\mu=\ell(\mu),
\]
and, for every \(r\), the union
\[
        \Sym^n_{\le r}\C:=\bigcup_{\ell(\mu)\le r}S_\mu
\]
is the closed locus of divisors whose support has cardinality at most \(r\).
In particular, \(S_{1^n}\) is the open configuration locus, and its complement
is the big diagonal.
\end{definition}

\begin{proposition}[Collision-general affine flags]
\label[proposition]{prop:collision-transverse-flags}
There is a nonempty Zariski open subset of the affine flag variety
parametrizing flags
\[
        \Lambda^0\subset\Lambda^1\subset\cdots\subset\Lambda^n=\Sym^n\C,
        \qquad \dim\Lambda^i=i,
\]
such that, for every partition \(\mu\vdash n\) and every \(i\), the
intersection \(\Lambda^i\cap S_\mu\) has the expected behavior:
\[
        \Lambda^i\cap S_\mu=\emptyset
        \qquad\text{if}\qquad
        i+\ell(\mu)<n,
\]
while, if \(i+\ell(\mu)\ge n\), then \(\Lambda^i\cap S_\mu\) is nonempty,
smooth, and of dimension
\[
        i+\ell(\mu)-n.
\]
In the second case the intersection is transverse, i.e. for every
\(x\in\Lambda^i\cap S_\mu\),
\[
        T_x\Lambda^i+T_xS_\mu=T_x\Sym^n\C .
\]
\end{proposition}

\begin{proof}
The collision stratification has finitely many smooth strata.  Fix \(i\) and
\(\mu\), and set
\[
        c:=n-i,
        \qquad d:=\dim S_\mu=\ell(\mu).
\]
An affine \(i\)-plane in \(\A^n\) may be viewed as a fiber of an affine linear
projection
\[
        \pi:\A^n\longrightarrow\A^c.
\]
For a general choice of \(\pi\), the restriction
\[
        \pi|_{S_\mu}:S_\mu\longrightarrow\A^c
\]
has the expected behavior.  If \(d<c\), then its image has dimension at most
\(d\), so a general fiber of \(\pi\) is disjoint from \(S_\mu\).  If \(d\ge c\),
then a general projection is dominant on \(S_\mu\); by generic smoothness, a
general fiber of \(\pi|_{S_\mu}\) is nonempty and smooth of dimension
\[
        d-c=\ell(\mu)+i-n.
\]
This smoothness of the fiber is equivalent to the transversality of the
corresponding affine \(i\)-plane with \(S_\mu\).

Thus, for each pair \((i,\mu)\), the affine \(i\)-planes with the stated
incidence and transversality property form a nonempty Zariski open subset of
the affine Grassmannian.  Pulling these open subsets back to the affine flag
variety and intersecting over the finitely many pairs \((i,\mu)\) gives a
nonempty Zariski open subset of the affine flag variety.  Any flag in this
open subset has the required property for every \(i\) and every \(\mu\).
\end{proof}

\begin{definition}[Good flags]
\label[definition]{def:dcm-generality}
In de Cataldo--Migliorini, a "sufficiently general" flag is such that it is in a good position with respect to a
stratification adapted to the constructible complex under consideration; a general flag satisfies this condition \cite[Def.~5.2.4 and
Rmk~5.2.5]{dCM}.  
In the application to the map considered in this paper
\[
        h_n:M_n\to \Sym^n\C\simeq\mathbb A^n,
\]
we choose the flag inside this dCM-open subset and, in addition, satisfying
the incidence and transversality properties with respect to the collision
stratification recorded in \cref{prop:collision-transverse-flags}.  We will
call such an affine linear flag a \textbf{good flag}.
\end{definition}

Thus, for the map considered in this paper
\[
        h_n:M_n(\Sigma)\to\Sym^n\C,
\]
we have a corollary of \cref{thm:dcm-flag}:
\begin{cor} 
\label[cor]{cor:indexing}
Given a good flag
\[
        \Lambda^{-1}=\emptyset\subset\Lambda^0\subset\Lambda^1\subset\cdots
        \subset\Lambda^n=\Sym^n\C
\]
we have the following equality
\[
        P_kH^{2n}(M_n(\Sigma))
        =
        \ker\left(
        H^{2n}(M_n(\Sigma))
        \longrightarrow
        H^{2n}\bigl(h_n^{-1}(\Lambda^{2n-k-1})\bigr)
        \right).
\]
\end{cor}

\section{Restrictions and branch counts}

\begin{lemma}[Pure-dimensional fibers]
\label[lemma]{lem:fibers-equidimensional}
Let
\[
        y=\mu_1z_1+\cdots+\mu_sz_s\in \Sym^n\C,
        \qquad z_a\neq z_b,
\]
with \(\mu=(\mu_1,\ldots,\mu_s)\vdash n\).  Then
\[
        h_n^{-1}(y)
        \simeq
        \prod_{a=1}^s F_{\mu_a},
        \qquad
        F_m:=h_m^{-1}(m\cdot0)
        \subset \Hilb^m(\Sigma\times\C).
\]
In particular, every fiber of
\[
        h_n:M_n(\Sigma)\to\Sym^n\C
\]
is pure-dimensional of complex dimension \(n\).  
\end{lemma}

\begin{proof}
Every subscheme lying over the divisor \(y\) decomposes uniquely into its
parts supported over the distinct vertical fibers
\[
        \Sigma\times\{z_1\},\ldots,\Sigma\times\{z_s\}.
\]
This gives the displayed product decomposition
\[
        h_n^{-1}(y)\simeq\prod_{a=1}^sF_{\mu_a}.
\]
For each \(m\), the fiber
\[
        F_m=h_m^{-1}(m\cdot0)
\]
is the core of \(M_m(\Sigma)\).  By \cref{lem:core-components}, applied with
\(n=m\), this core is pure-dimensional of complex dimension \(m\).  Therefore
\(h_n^{-1}(y)\) is pure-dimensional of dimension
\[
        \sum_{a=1}^s\mu_a=n.\qedhere
\]
\end{proof}

The first part of the following lemma follows from the previous work of {\MZ}
\cite[pf. of Cor.~4.15]{MZ}, but we prove it here for completeness of the exposition.

\begin{lemma}[Local sheets of upward-flow cycles]
\label[lemma]{lem:vertical-sheets}
Let
\[
        \lambda=1^{\alpha_1}2^{\alpha_2}\cdots
\]
be a partition of \(n\), and let \(p\in F_\lambda^\circ\).  Write
\[
        p=(D_i)_{i\ge1},
        \qquad
        D_i=x_{i,1}+\cdots+x_{i,\alpha_i},
\]
with all \(x_{i,a}\in\Sigma\) pairwise distinct.  Then, after choosing this
point \(p\), the upward-flow variety \(W_p^+\) is naturally identified with
the space of divisors moving only in the vertical curves
\[
        \{x_{i,a}\}\times\C .
\]
More precisely, there is an identification
\[
        W_p^+
        \cong
        \prod_{i\ge1}\prod_{a=1}^{\alpha_i}\Sym^i\C,
\]
and under this identification the restriction
\[
        h_n|_{W_p^+}:W_p^+\longrightarrow \Sym^n\C
\]
is the symmetrization map
\[
        \prod_{i\ge1}\prod_{a=1}^{\alpha_i}\Sym^i\C
        \longrightarrow
        \Sym^n\C,
        \qquad
        (D_{i,a})_{i,a}\longmapsto \sum_{i,a}D_{i,a}.
\]
Consequently, over a point
\[
        y=\mu_1z_1+\cdots+\mu_sz_s\in S_\mu,
        \qquad z_b\neq z_{b'},
\]
a local sheet is determined by a matrix of nonnegative integers
\[
        (m_{i,a;b})
\]
such that
\[
        \sum_b m_{i,a;b}=i
        \quad\text{for every }(i,a),
        \qquad
        \sum_{i,a}m_{i,a;b}=\mu_b
        \quad\text{for every }b.
\]
The entry \(m_{i,a;b}\) records how much of the divisor over \(z_b\) lies on
the fixed vertical curve \(\{x_{i,a}\}\times\C\).
\end{lemma}

\begin{proof}
Let
\[
        p=(D_i)_{i\ge1}\in F_\lambda^\circ,
        \qquad
        D_i=x_{i,1}+\cdots+x_{i,\alpha_i},
\]
with all points \(x_{i,a}\) pairwise distinct.  By the description of the
fixed component \(F_\lambda\), the point \(p\) is a product of fixed punctual
subschemes, each supported at one of the points $(x_{i,a},0)\in\Sigma\times\C$
on the corresponding vertical curve $\{x_{i,a}\}\times\C$.
Since the support points \((x_{i,a},0)\) are pairwise distinct, the Hilbert
scheme is locally, near these subschemes, the product of the corresponding
punctual Hilbert schemes.  Under this product description the \(\C^*\)-action
is the product of the standard scaling actions on the vertical coordinates.

For a single fixed punctual factor of length \(i\) supported at \((x,0)\),
the attracting set is precisely
\[
        \Sym^i(\{x\}\times\C)\cong \Sym^i\C.
\]
Indeed, the attracting directions are obtained by allowing the vertical
\(\C\)-coordinates to move along the vertical curve \(\{x\}\times\C\)
while keeping the \(\Sigma\)-coordinate equal to \(x\).  Taking the product over all \((i,a)\) gives
\[
        W_p^+
        \cong
        \prod_{i\ge1}\prod_{a=1}^{\alpha_i}\Sym^i\C.
\]

The map \(h_n\) records only the projection of the subscheme to \(\C\). Hence
a collection of divisors
\[
        (D_{i,a})_{i,a}\in
        \prod_{i\ge1}\prod_{a=1}^{\alpha_i}\Sym^i\C
\]
is sent to their sum
\[
        \sum_{i,a}D_{i,a}\in\Sym^n\C,
\]
which is exactly the displayed symmetrization map.

Finally, let
\[
        y=\mu_1z_1+\cdots+\mu_sz_s\in S_\mu,
        \qquad z_b\neq z_{b'}.
\]
A point of the reduced fiber over \(y\) is obtained by distributing, for each
vertical curve \(\{x_{i,a}\}\times\C\), a divisor of degree \(i\) among the
support points \(z_1,\ldots,z_s\).  Thus it is specified by nonnegative
integers \(m_{i,a;b}\) satisfying
\[
        \sum_b m_{i,a;b}=i,
        \qquad
        \sum_{i,a}m_{i,a;b}=\mu_b. \qedhere
\]
\end{proof}

For a good\footnote{Recall \cref{def:dcm-generality}} affine linear subspace
\[
        L\subset \Sym^n\C
\]
we write
\[
        M_L:=h_n^{-1}(L),
        \qquad
        i_L:M_L\hookrightarrow M_n
\]
for the inverse image and its inclusion.  If \(y\in L\), we write
\[
        M_y:=h_n^{-1}(y),
        \qquad
        i_y:M_y\hookrightarrow M_n,
        \qquad
        i_{y,L}:M_y\hookrightarrow M_L.
\]
Thus
\[
        i_y=i_L\circ i_{y,L}.
\]
By \cref{lem:fibers-equidimensional}, the fiber \(M_y\) is pure of
complex dimension \(n\).  Hence its top Borel--Moore homology is generated
by the fundamental classes of the irreducible components of
\(M_y\).  In particular, if
\(\beta\in H^{2n}(M_y)\), then one may pair \(\beta\) with each such
component, and also with the reduced fundamental class of its reduced scheme structure (the scheme structure coming from $h_n^{-1}(y)$ is generally non-reduced)
\[
        [(M_y)_{\red}]
        :=
        \sum_{C\in \operatorname{Irr}(M_y)}[C]
        \in H^{\BM}_{2n}(M_y).
\]

\begin{lemma}[Reduced fiber pairing]
\label[lemma]{lem:reduced-fiber-pairing}
Let \(y\in S_\mu\).  Then, choosing the representative
\(W_\lambda^+=W_p^+\) with \(p\in F_\lambda^\circ\), one has
\[
        \left\langle i_y^*U_\lambda,[(M_y)_{\mathrm{red}}]\right\rangle
        =
        n_\lambda^\mu .
\]
\end{lemma}

\begin{proof}
Write
\[
        \lambda=1^{\alpha_1}2^{\alpha_2}\cdots .
\]
Choose
\[
        p=(D_i)_{i\ge 1}\in F_\lambda^\circ,
        \qquad
        D_i=x_{i,1}+\cdots+x_{i,\alpha_i},
\]
thus all support points \(x_{i,a}\in\Sigma\) are pairwise distinct.  

By \cref{lem:vertical-sheets}, the intersection
\[
        W_p^+\cap M_y
        =
        (h_n|_{W_p^+})^{-1}(y)
\]
is the underlying set of branches of the corresponding symmetrization map,
i.e.\ a finite set of points.  These points are indexed by nonnegative
integer matrices with row sums determined by \(\lambda\) and column sums
determined by \(\mu\).  By
\cref{lem:matrix-count}, the number of such points is \(n_\lambda^\mu\).

We now explain why this reduced count is the relevant pairing.  Let
\(M_{y,a}\) be an irreducible component of \((M_y)_{\mathrm{red}}\).  Since
\(U_\lambda=\PD[W_\lambda^+]\) on the smooth variety \(M_n\), we have
\[
        \bigl\langle i_y^*U_\lambda,[M_{y,a}]\bigr\rangle
        =  \bigl\langle U_\lambda,{i_y}_*[M_{y,a}]\bigr\rangle
        =[W_\lambda^+]\cdot {i_y}_*[M_{y,a}].
\]
Thus pairing with
\[
        [(M_y)_{\mathrm{red}}]=\sum_a[M_{y,a}]
\]
is the sum of these ordinary intersection numbers with the irreducible
components of the reduced fiber.

We now identify the component of \((M_y)_{\mathrm{red}}\)
containing each intersection point.  By \cref{lem:fibers-equidimensional}
and \cref{lem:core-components} applied to each factor, the irreducible
components of \((M_y)_{\mathrm{red}}\) are the products
\(\prod_a C_{\nu_a}\) with \(\nu_a\vdash\mu_a\), where
\(C_{\nu_a}=\overline{W^-_{F_{\nu_a}}}\subset F_{\mu_a}\).  An intersection
point sits, in each factor \(F_{\mu_a}\), at a \(\C^*\)-fixed configuration
of curvilinear vertical subschemes on the pairwise distinct vertical curves
\(\{x_{i,a}\}\times\C\); it is therefore a fixed point of
\(F_{\nu_a}\subset C_{\nu_a}\), where \(\nu_a\) is the multiset of vertical
lengths \(\{m_{i,a;b}: m_{i,a;b}>0\}\).  Since the \(x_{i,a}\) are pairwise
distinct, \(\nu_a\) already has the maximal number of parts compatible with
these local lengths; passing to any other \(C_{\nu'_a}\) in its closure
would require a collision of \(\Sigma\)-supports, which our intersection
point does not have.  Hence each intersection point lies in the smooth
locus of a unique component of \((M_y)_{\mathrm{red}}\).

In local product coordinates around such a point, the smooth surface
\(\Sigma\times\C\) factors as
\[
        \Sigma\times\C
        =
        \underbrace{\Sigma}_{\text{horizontal}}\times\underbrace{\C}_{\text{vertical}},
\]
and the Hilbert scheme \(M_n\) inherits a corresponding local splitting of
tangent directions into
\[
        T_{[Z]}M_n
        =
        T^{\Sigma}_{[Z]}\oplus T^{\C}_{[Z]} ,
\]
with \(\dim_\C T^{\Sigma}_{[Z]}=n\) coming from horizontal deformations of
support points and \(\dim_\C T^{\C}_{[Z]}=n\) coming from vertical
deformations of the curvilinear pieces.  By \cref{lem:vertical-sheets},
\(W_p^+\) is locally cut out by fixing the \(\Sigma\)-coordinates of the
support, so
\[
        T_{[Z]}W_p^+= T^{\C}_{[Z]}.
\]
Dually, the relevant component \(M_{y,a}\) of \((M_y)_{\mathrm{red}}\) is
locally cut out by fixing the vertical projection (i.e.\ the image under
\(h_n\)), so
\[
        T_{[Z]}M_{y,a}= T^{\Sigma}_{[Z]}.
\]
These two tangent subspaces are complementary, and their dimensions sum to
\(2n=\dim_\C M_n\); hence \(W_p^+\) and \(M_{y,a}\) meet transversely at
\([Z]\).  With the complex orientations, each local intersection
multiplicity is \(+1\).  Summing over the \(n_\lambda^\mu\) intersection
points,
\[
        \left\langle i_y^*U_\lambda,[(M_y)_{\mathrm{red}}]\right\rangle
        =
        \#\bigl((h_n|_{W_p^+})^{-1}(y)\bigr)_{\mathrm{red}}
        =
        n_\lambda^\mu. \qedhere
\]
\end{proof}

\begin{rmk}(Multiplicities of core components)
Due to  \cref{cor: fla za nlambda1n} we have explicit formula
\[
        n_\lambda^{1^n}
        =
        \frac{n!}{\prod_i\lambda_i!}.
\]
This is the generic-fiber intersection number with $U_\lambda$
\[
        W_\lambda^+\cap h_n^{-1}(y)
        =
        \frac{n!}{\prod_i\lambda_i!},
        \qquad y\in S_{1^n}=\Sym^n\C\setminus \Delta.
\]
These numbers also record the algebraic multiplicity of the corresponding irreducible core component
$$C_\lambda := \overline{\{ p\in M_n \mid \lim_{t\fun \infty} t\cdot p \in F_\lambda\}} $$
seen as a component of the scheme $h_n^{-1}(0)$. The latter are
computed\footnote{For the $\Sigma=E$ case, but the general case works verbatim} in  \cite[Cor.~4.14]{MZ}, and the equality between the two 
is a general principle from \cite[Thm~1.3]{hausel2022very}. 
\end{rmk}

\begin{proposition}[Restriction implies incidence equations]
\label[proposition]{prop:restriction-implies-incidence}
Let \(L\subset\Sym^n\C\) be a good\footnote{Recall \cref{def:dcm-generality}.} affine linear subspace of dimension
\(d\), and let
\[
        \alpha=\sum_{\lambda\vdash n}c_\lambda U_\lambda
        \in H^{2n}(M_n).
\]
If
\[
        i_L^*\alpha=0\in H^{2n}(M_L),
\]
then
\[
        \sum_{\lambda\vdash n}c_\lambda n_\lambda^\mu=0
\]
for every partition \(\mu\vdash n\) such that
\[
        L\cap S_\mu\neq\emptyset .
\]
Equivalently, by the incidence property of a good affine subspace, this holds
for every \(\mu\) satisfying
\[
        \ell(\mu)+d-n\ge 0.
\]
\end{proposition}

\begin{proof}
Let \(\mu\vdash n\) be such that \(L\cap S_\mu\neq\emptyset\), and choose a
general point
\[
        y\in L\cap S_\mu .
\]
Since \(i_y=i_L\circ i_{y,L}\), ordinary cohomology pullback gives
\[
        i_y^*\alpha
        =
        i_{y,L}^*i_L^*\alpha
        =0
        \in H^{2n}(M_y).
\]
Pairing with the reduced fundamental class of the fiber gives
\[
        0
        =
        \left\langle i_y^*\alpha,[(M_y)_{\mathrm{red}}]\right\rangle
        =
        \sum_{\lambda\vdash n}c_\lambda
        \left\langle i_y^*U_\lambda,[(M_y)_{\mathrm{red}}]\right\rangle.
\]
By \cref{lem:reduced-fiber-pairing}, the last expression is
\[
        \sum_{\lambda\vdash n}c_\lambda n_\lambda^\mu. \qedhere
\]

\end{proof}

\section{Perverse filtration on \(H^{2n}(M_n)\)}

\begin{proposition}[Incidence equations from the dCM criterion]
\label[proposition]{prop:support-dcm}
Let
\[
        \alpha=\sum_{\lambda\vdash n}c_\lambda U_\lambda \in P_kH^{2n}(M_n).
\]
Then
\[
        \sum_{\lambda\vdash n}c_\lambda n_\lambda^\mu=0
        \qquad
        \text{for all }\mu\vdash n\text{ such that }
        \ell(\mu)>k-n.
\]
\end{proposition}

\begin{proof}
The boundary case \(k\ge 2n\) is immediate, since there is no partition
\(\mu\vdash n\) with \(\ell(\mu)>k-n\).  If \(k<n\), then
\(2n-k-1\ge n\), so the dCM restriction is the identity map on
\(H^{2n}(M_n)\).  Hence \(P_kH^{2n}(M_n)=0\), and the statement is vacuous.
Thus we may assume \(n\le k<2n\).

Choose a good flag
\[
        \Lambda^0\subset\Lambda^1\subset\cdots\subset\Lambda^n=\Sym^n\C,
        \qquad
        \dim\Lambda^i=i.
\]
By \cref{cor:indexing}, condition \(\alpha\in P_kH^{2n}(M_n)\) implies that
\(\alpha\) restricts trivially to
\[
        M_L:=h_n^{-1}(L),
        \qquad
        L:=\Lambda^{2n-k-1}.
\]
Applying \cref{prop:restriction-implies-incidence}, we get
\[
        \sum_{\lambda\vdash n}c_\lambda n_\lambda^\mu=0
\]
for every 
\(\mu\) with

\[
        \ell(\mu)+n-k-1\ge0,
\]
or equivalently
\[
        \ell(\mu)>k-n. \qedhere
\]
\end{proof}

We now combine the incidence equations with the symmetric-function calculation.  By \cref{lem:upward-flow-basis}, the classes \(\{U_\mu\}\) form
a basis of
\[
        V_n:=H^{2n}(M_n).
\]
By \cref{lem:h-basis}, the complete homogeneous symmetric functions \(\{h_\mu\}\) form a basis of
\(\Lambda^n\).  We may therefore define a linear isomorphism
\[
        \psi:\Lambda^n\xrightarrow{\;\iso\;}V_n,
        \qquad
        h_\mu\longmapsto U_\mu,
\]
which we will refer to as the \emph{geometric dictionary}. 
Recalling the products of Newton power sums $p_\lambda=p_{\lambda_1}\cdot \dots \cdot p_{\lambda_{\ell(\lambda)}}$, set
\[
        \eta_\lambda:=\psi(p_\lambda).
\]
Equivalently, with the coefficients \(a_{\lambda\mu}\) of \cref{def:a-lambda-mu},
\[
        \eta_\lambda=\sum_{\mu\vdash n}a_{\lambda\mu}U_\mu.
\]

The main theorem of the paper is the following.  

\begin{theorem}\label[theorem]{thm:main}
For every \(k\),
\[
        P_kH^{2n}(M_n)
        =
        \Span\{\eta_\lambda \mid \ell(\lambda)\le k-n\}.
\]
Equivalently,
\[
        \Gr^P_kH^{2n}(M_n)
        =
        \Span\{[\eta_\lambda]\mid \ell(\lambda)=k-n\}.
\]
\end{theorem}

\begin{proof}
Let
\[
        V_n:=H^{2n}(M_n),
\]
and define
\[
        K_k:=
        \left\{
        \sum_{\lambda\vdash n}c_\lambda U_\lambda
        \;\middle|\;
        \sum_{\lambda\vdash n}c_\lambda n_\lambda^\mu=0
        \text{ for all }\mu\text{ with }\ell(\mu)>k-n
        \right\}.
\]
By \cref{prop:support-dcm}, we have an inclusion
\[
        P_kH^{2n}(M_n)\subseteq K_k.
\]

We now identify \(K_k\) using the symmetric-function dictionary
\(\psi(h_\lambda)=U_\lambda\).  Let
\[
        f=\sum_{\lambda\vdash n}c_\lambda h_\lambda\in \Lambda^n
\]
be the symmetric function corresponding to
\(\sum_\lambda c_\lambda U_\lambda\).  By the definition of the coefficients
\(n_\lambda^\mu\), one has
\[
        h_\lambda=\sum_{\mu\vdash n} n_\lambda^\mu m_\mu .
\]
Therefore the coefficient of \(m_\mu\) in \(f\) is
\[
        \sum_{\lambda\vdash n}c_\lambda n_\lambda^\mu .
\]
Thus the defining equations of \(K_k\) are exactly the equations saying that
all monomial coefficients of \(f\) indexed by partitions
\(\mu\) with
\[
        \ell(\mu)>k-n
\]
vanish.  Equivalently, \(K_k\) is the image under \(\psi\) of the subspace of
\(\Lambda^n\) spanned by the monomial symmetric functions
\(m_\mu\) with \(\ell(\mu)\le k-n\).  Since the monomial symmetric functions form
a basis of \(\Lambda^n\), these equations cut out a subspace of dimension
\[
        \#\{\lambda\vdash n\mid \ell(\lambda)\le k-n\}=\dim K_k.
\]
Let us now check explicitly that the proposed power-sum classes lie in
\(K_k\).  Fix a partition \(\lambda\vdash n\) with
\[
        \ell(\lambda)\le k-n.
\]
By definition,
\[
        \eta_\lambda=\psi(p_\lambda).
\]
For a forbidden partition \(\mu\), meaning
\[
        \ell(\mu)>k-n,
\]
we also have
\[
        \ell(\mu)>\ell(\lambda).
\]
By \cref{lem:p-support-length}, no monomial symmetric function indexed by
such a \(\mu\) can occur in the monomial expansion of \(p_\lambda\).  Thus the
\(\mu\)-coefficient of \(p_\lambda\) is zero for every forbidden \(\mu\).  In
the dictionary \(\psi(h_\lambda)=U_\lambda\), the defining equations of
\(K_k\) are precisely the vanishing of these forbidden monomial coefficients.
Hence
\[
        \eta_\lambda\in K_k.
\]
Therefore
\[
        \Span\{\eta_\lambda\mid \ell(\lambda)\le k-n\}
        \subseteq K_k.
\]
Polynomials \(p_\lambda\) are linearly independent by
\cref{lem:p-basis}, thus so are classes $\eta_\lambda$. 
Since their number in $K_k$ is equal to its dimension, we
obtain
\[
        K_k=
        \Span\{\eta_\lambda\mid \ell(\lambda)\le k-n\}.
\]

Finally, by the rank formula of Shen--Zhang \cite[Prop~1.2]{ShenZhang}
\[
        \dim P_kH^{2n}(M_n)
        =
        \#\{\lambda\vdash n\mid \ell(\lambda)\le k-n\}
        =
        \dim K_k.
\]
The inclusion \(P_kH^{2n}(M_n)\subseteq K_k\) is therefore an equality.  This
proves
\[
        P_kH^{2n}(M_n)
        =
        \Span\{\eta_\lambda\mid \ell(\lambda)\le k-n\}.
\]
The statement for the associated graded follows immediately by taking the
quotient
\[
        P_kH^{2n}(M_n)/P_{k-1}H^{2n}(M_n).
\qedhere
\]
\end{proof}

\subsection{The case \(n=2\)}\label{ssec:n-equals-2}

We illustrate the dictionary in the smallest non-trivial case.  Order
partitions of \(2\) by
\[
        2,\quad 11.
\]
The incidence matrix $N=(n_\lambda^\mu)$ is
\[
\begin{array}{c|cc}
\mu\backslash \lambda
&2&11\\
\hline
2  &1&1\\
11 &1&2
\end{array}
\]
Thus by \cref{prop:support-dcm}, solving equation for $\mu=(11)$
$$c_2 n_2^{11} + c_{11}n_{11}^{11} = c_2 \cdot 1+ c_{11}\cdot 2 =0$$
gives us directly
$$P_3=\langle \sum c_\lambda U_\lambda \rangle =\langle 2U_2-U_{11}\rangle$$
On the other hand, the power-sum expansions in basis of complete homogeneous symmetric functions
are
\[
        p_2=2h_2-h_{11},
        \qquad
        p_{11}=h_{11}.
\]
Therefore
\[
        \eta_2=2U_2-U_{11},
        \qquad
        \eta_{11}=U_{11},
\]
and the perverse filtration by \cref{thm:main} is
\[
        P_3H^4(M_2)=\langle \eta_2\rangle,
        \qquad
        P_4H^4(M_2)=\la \eta_2, \eta_{11} \ra = H^4(M_2).
\]
The class \(\eta_2=2U_2-U_{11}\) is the unique (up to scalar) combination of
upward-flow classes whose intersection with a generic fiber of \(h_2\)
vanishes.

\subsection{The case \(n=3\)}

Continuing the computation of \cref{ex:n-equals-3} via the dictionary
\(\psi\), we obtain
\[
        \eta_3=3U_3-3U_{21}+U_{111},
        \qquad
        \eta_{21}=2U_{21}-U_{111},
        \qquad
        \eta_{111}=U_{111}.
\]
Hence 
\[
        P_4H^6(M_3)=\langle\eta_3\rangle,
        \qquad
        P_5H^6(M_3)=\langle\eta_3,\eta_{21}\rangle,
        \qquad
        P_6H^6(M_3)=H^6(M_3).
\]

\subsection{The case \(n=4\)}

Write partitions in the order
\[
        4,\quad 31,\quad 22,\quad 211,\quad 1111.
\]
The incidence matrix \(N=(n_\lambda^\mu)\), with rows indexed by collision
types \(\mu\) and columns indexed by upward-flow classes \(U_\lambda\), is
\[
\begin{array}{c|ccccc}
\mu\backslash \lambda
&4&31&22&211&1111\\
\hline
4      &1&1&1&1&1\\
31     &1&2&2&3&4\\
22     &1&2&3&4&6\\
211    &1&3&4&7&12\\
1111   &1&4&6&12&24
\end{array}
\]
This gives us the equations for perverse filtration. For example, to obtain 
$P_6$ one has to solve equations that correspond to partitions $\mu$
with $l(\mu)>6-n=2$. Therefore corresponding to $\mu=211$ and $1111$, we get respectively:
\[
        c_4+3c_{31}+4c_{22}+7c_{211}+12c_{1111}=0, 
        \] \[
        c_4+4c_{31}+6c_{22}+12c_{211}+24c_{1111}=0.
\]
The space of all solutions $c_\lambda$ gives us $P_6=\{\sum_\lambda c_\lambda U_\lambda\}$.

On the other hand, the Newton power sums expansion in the 
\(h\)-basis 
yields
\[
\begin{pmatrix}
\eta_4\\
\eta_{31}\\
\eta_{22}\\
\eta_{211}\\
\eta_{1111}
\end{pmatrix}
=
\begin{pmatrix}
4&-4&-2&4&-1\\
0&3&0&-3&1\\
0&0&4&-4&1\\
0&0&0&2&-1\\
0&0&0&0&1
\end{pmatrix}
\begin{pmatrix}
U_4\\
U_{31}\\
U_{22}\\
U_{211}\\
U_{1111}
\end{pmatrix},
\]
Thus the perverse filtration is
\[
        P_5H^8(M_4)
        =
        \langle \eta_4\rangle,
\ \
        P_6H^8(M_4)
        =
        \langle \eta_4,\eta_{31},\eta_{22}\rangle,
\ \
        P_7H^8(M_4)
        =
        \langle \eta_4,\eta_{31},\eta_{22},\eta_{211}\rangle,
\ \
        P_8H^8(M_4)=H^8(M_4).
\]

\begin{remark}[Relation with the Schur basis]
Under the symmetric-function dictionary
\[
        U_\mu\leftrightarrow h_\mu,
        \qquad
        \eta_\lambda\leftrightarrow p_\lambda,
\]
one may also introduce classes $\sigma_\nu \in H^{2n}(M_n)$ corresponding to the Schur
functions \(s_\nu\).  Then the Frobenius character formula gives
\[
        \eta_\lambda
        =
        \sum_{\nu\vdash n}\chi^\nu(\lambda)\sigma_\nu,
\]
where \(\chi^\nu(\lambda)\) is the value of the irreducible character of
\(S_n\) indexed by \(\nu\) on the conjugacy class of cycle type \(\lambda\).
Thus the transition matrix from the perverse to this Schur basis is the
character table of \(S_n\).
\end{remark}

\bibliographystyle{amsalpha}
\bibliography{FZ}

\end{document}